\documentclass[a4paper,11pt,reqno]{amsart}
\usepackage{amsmath}
\usepackage{amssymb}
\usepackage{xcolor}
\usepackage{comment}
\usepackage{enumerate}

\theoremstyle{definition} 
\newtheorem{theorem}{Theorem}[section]

\newtheorem{lemma}[theorem]{Lemma}
\newtheorem{corollary}[theorem]{Corollary}

\makeatletter
    
    \@addtoreset{equation}{section}
  \makeatother

\title[Bernoulli first-passage percolation]{Comparison of limit shapes for Bernoulli first-passage percolation}
\thanks{N.K. was supported by JSPS KAKENHI Grant Number JP20K14332. M.T. is partially supported by JSPS KAKENHI Grant Numbers JP19H01793, JP19K03514 and JP22K03333.}
\author{Naoki Kubota}
\address{College of Science and Technology, Nihon University}
\email{kubota.naoki08@nihon-u.ac.jp}
\author{Masato Takei}
\address{Department of Applied Mathematics, Faculty of Engineering, Yokohama National University, Yokohama, Japan}
\email{takei-masato-fx@ynu.ac.jp}

\begin{document}

\begin{abstract}
We consider Bernoulli first-passage percolation on the $d$-dimensional hypercubic lattice with $d \geq 2$. The passage time of edge $e$ is $0$ with probability $p$ and $1$ with probability $1-p$, independently of each other.
Let $p_c$ be the critical probability for percolation of edges with passage time $0$. When $0\leq p<p_c$, there exists a nonrandom, nonempty compact convex set $\mathcal{B}_p$ such that the set of vertices to which the first-passage time from the origin is within $t$ is well-approximated by $t\mathcal{B}_p$ for all large $t$, with probability one. The aim of this paper is to prove that for $0\leq p<q<p_c$, the Hausdorff distance between $\mathcal{B}_p$ and $\mathcal{B}_q$ grows linearly in $q-p$. Moreover, we mention that the approach taken in the paper provides a lower bound for the expected size of the intersection of geodesics, that gives a nontrivial consequence for the {\it critical} case.  
\end{abstract}

\maketitle

\section{Introduction}
\label{intro}

First-passage percolation was first introduced by Hammersley and Welsh \cite{HammersleyWelsh65}. 
In this paper, let $d \geq 2$ and consider Bernoulli first-passage percolation on the $d$-dimensional hypercubic lattice $\mathbb{L}^d = (\mathbb{Z}^d,\mathbb{E}^d)$: The family $\{ t(e) \}_{e \in \mathbb{E}^d}$ of independent random variables satisfies that
\[ P_p(t(e)=0)=p \quad \mbox{and}\quad P_p(t(e)=1)=1-p \quad \mbox{for each $e \in \mathbb{E}^d$.}\]
The expectation of a random variable $X$ with respect to $P_p$ is denoted by $E_p[X]$. For a self-avoiding path $\gamma$ from $\boldsymbol{x} \in \mathbb{Z}^d$ to $\boldsymbol{y} \in \mathbb{Z}^d$, we define  
\[
t(\gamma) := \sum_{e \in \gamma} t(e).
\]
For each $\boldsymbol{x} \in \mathbb{R}^d$, we can find a unique $\boldsymbol{x'} \in \mathbb{Z}^d$ with $\boldsymbol{x} \in \boldsymbol{x'} + [0,1)^d$. The first-passage time from $\boldsymbol{x} \in \mathbb{R}^d$ to $\boldsymbol{y} \in \mathbb{R}^d$ is defined by
\[
T(\boldsymbol{x},\boldsymbol{y}) := \inf_{\gamma : \boldsymbol{x}'\to\boldsymbol{y}'} t(\gamma).
\]
The wet region at time $t \geq 0$ is defined by 
\[
B(t):=\{ \boldsymbol{y} \in \mathbb{R}^d : T(\boldsymbol{O},\boldsymbol{y}) \leq t \},
\]
where $\boldsymbol{O}:=(0,0,\ldots,0)$. Let
\[ p_c=p_c(d) := \inf
\left\{ p \in [0,1] :
\begin{array}{@{\,}c@{\,}}
\text{there exists an infinite self-avoiding path}\\
\text{of edges $e$ with $t(e)=0$} \\
\end{array}
\right\} .
\]
It is known that when $0 \leq p<p_c$, there exists a nonrandom, nonempty compact convex set $\mathcal{B}_p$ such that for any $\varepsilon > 0$,
\begin{align}
P_p\left (
\text{$(1-\varepsilon) \mathcal{B}_p \subset \dfrac{1}{t} B(t) \subset (1+\varepsilon) \mathcal{B}_p$ for all large $t$}
\right)=1.
\label{thm:ShapeTheoremBernoulli}
\end{align}
This is called the shape theorem, a kind of law of large numbers for the wet region. For fundamental results on first-passage percolation (including the shape theorem above), refer to Kesten \cite{Kesten86LNM} and Auffinger, Damron and Hanson \cite{AuffingerDamronHanson50years}.

In this paper, we compare limit shapes $\mathcal{B}_p$, $0 \leq p < p_c$, and prove that a certain distance between $\mathcal{B}_p$ and $\mathcal{B}_q$ is of order $|q-p|$. Moreover, the method used to compare limit shapes also gives a nontrivial property for the size of the intersection of paths realizing the first-passage time.

\section{Results}

For $p \in [0,1]$ and $\boldsymbol{x} \in \mathbb{Z}^d$, the following limit exists (see Chapter 2 in \cite{Kesten86LNM}):
\begin{align}
\mu_p(\boldsymbol{x}) := \lim_{n \to \infty} \dfrac{E_p[T(\boldsymbol{O},n\boldsymbol{x})]}{n} =\inf_{n \geq 1} \dfrac{E_p[T(\boldsymbol{O},n\boldsymbol{x})]}{n}.
\label{def:TimeConstGen}
\end{align}
It is known that $\mu_p(\boldsymbol{x})$ can be extended to a seminorm on $\mathbb{R}^d$, and $\mu_p(\boldsymbol{x})$ is called the time constant. In particular, for $0 \leq p<p_c$, $\mu_p(\boldsymbol{x})>0$ for any $\boldsymbol{x} \in \mathbb{R}^d \setminus \{ \boldsymbol{O}\}$, and the limit shape $\mathcal{B}_p$ in \eqref{thm:ShapeTheoremBernoulli} is the unit ball for the norm $\mu_p(\,\cdot\,)$:
\[ \mathcal{B}_p = \{ \boldsymbol{x} \in \mathbb{R}^d : \mu_p(\boldsymbol{x}) \leq 1 \}. \]
Note that if $0\leq p<q \leq 1$, then $\mu_p(\boldsymbol{x}) \geq \mu_q(\boldsymbol{x})$ for any $\boldsymbol{x} \in \mathbb{R}^d$. 
For $p \geq p_c$, it is known that $\mu_p(\boldsymbol{x})=0$ for any $\boldsymbol{x} \in \mathbb{R}^d$, and the limit shape $\mathcal{B}_p=\mathbb{R}^d$, in that for any compact $K \subset \mathbb{R}^d$, $K \subset t^{-1} B(t)$ for all large $t$ with probability one. We can see that $\mathcal{B}_p \subset \mathcal{B}_q$ if $0\leq p<q \leq 1$.

For $\boldsymbol{e}_1:=(1,0,\ldots,0)$, the general theory developed in \cite{Cox80,CoxKesten81,Kesten86LNM} tells us that $\mu_p(\boldsymbol{e}_1)$ is continuous in $p$. Moreover, from another general result by van den Berg and Kesten \cite{vandenBergKesten93AOAP}, $\mu_p(\boldsymbol{e}_1)$ is strictly decreasing in $p \in [0,p_c]$. Furthermore, Wu and Feng \cite{WuFeng09} derived the following concrete lower bound for the decrease: If $0 \leq p < q <p_c$, then
\begin{align}
 \mu_p(\boldsymbol{e}_1) - \mu_q(\boldsymbol{e}_1) \geq \dfrac{\mu_q(\boldsymbol{e}_1)}{1-q} \cdot (q-p).
 \label{ineq:WuFeng09}
\end{align}

Our first theorem extends \eqref{ineq:WuFeng09} to all directions.

\begin{theorem} \label{Thm:TimeConstStrict}
If $0 \leq p < q <p_c$, then for each $\boldsymbol{x} \in \mathbb{R}^d$,
\begin{align}
 \mu_p(\boldsymbol{x}) - \mu_q(\boldsymbol{x}) \geq \dfrac{\mu_q(\boldsymbol{x})}{1-q} \cdot (q-p).
 \label{Ineq:TimeConstStrict}
\end{align}
\end{theorem}

The next theorem gives a modulus of continuity of $\mu_p(\boldsymbol{x})$ in $p$.

\begin{theorem} \label{Thm:TimeConstLipschitz} Assume that $0<q_0<p_c$. There exists a constant $K=K(q_0)>0$ such that for any $0<p < q \leq q_0$ and any $\boldsymbol{x} \in \mathbb{R}^d$.
\begin{align}
\mu_p(\boldsymbol{x}) - \mu_q(\boldsymbol{x}) \leq \dfrac{K\mu_q(\boldsymbol{x})}{q}  \cdot (q-p) .
 \label{Ineq:TimeConstLipschitz}
\end{align}
In particular, $\mu_p(\boldsymbol{x})$ is Lipschitz continuous over any closed subinterval $[p_0,q_0]$ of $(0,p_c)$. 
\end{theorem}


For $A \subset \mathbb{R}^d$ and $\delta > 0$, set 
\[ N_{\delta} (A) := \{ \boldsymbol{y} \in \mathbb{R}^d : \text{$\| \boldsymbol{y}-\boldsymbol{x} \|_2 < \delta$ for some $\boldsymbol{x} \in A$} \}. \]
For $A,B \subset \mathbb{R}^d$, the Hausdorff distance between $A$ and $B$ is defined by
\[
d_{\mathcal{H}}(A,B) := \inf\{ \delta>0 : \text{$A \subset N_{\delta} (B)$ and $B \subset N_{\delta} (A)$} \}.
\]
As a corollary to Theorems \ref{Thm:TimeConstStrict} and \ref{Thm:TimeConstLipschitz}, we can derive linear bounds for $d_{\mathcal{H}}(\mathcal{B}_p,\mathcal{B}_q)$ in $|q-p|$.

\begin{corollary} \label{Cor:AsymptoticShapeHaudorff} Assume that $0<p_0<q_0<p_c$. We can find a constant $C>0$ such that for $p_0\leq p < q \leq q_0$, 
\[ q-p \leq d_{\mathcal{H}}(\mathcal{B}_p,\mathcal{B}_q) \leq C(q-p).\]
\end{corollary} 

We review literatures closely related to our main results. In bond percolation on $\mathbb{L}^d$ with density $\rho$ (i.e. each edge is open with probability $\rho$), if $\rho>p_c$, then there exists a unique infinite open cluster $\mathcal{C}_{\infty}$ with probability one. The first-passage percolation on $\mathcal{C}_{\infty}$ is studied by Garet, Marchand, Procaccia, and Th\'{e}ret \cite{GaretMarchandProcacciaTheret17EJP}. They show that the continuity of the time constant and the limit shape with respect to the probability law of $t(e)$. This model is equivalent to allow $t(e)$ to be $+\infty$ in first-passage percolation on $\mathbb{L}^d$. In the special case where
\[ P_{\rho} (t(e)=1)=\rho \quad \mbox{and} \quad P_{\rho} (t(e)=+\infty)=1-\rho, \]
the regularity of the time constant and the limit shape with respect to $\rho \in (p_c,1]$ is explored by  
Dembin \cite{Dembin21ESAIM} and Cerf and Dembin \cite{CerfDembin21}. It is shown in \cite{CerfDembin21} that for $p_0>p_c$, the time constant is Lipschitz continuous on $\rho \in [p_0,1]$. Note that although this result looks similar to our Theorem \ref{Thm:TimeConstLipschitz}, the results in \cite{CerfDembin21} do not imply it. 

The proofs of our main results will be given in Section \ref{sec:Proofs}, where geometric properties of geodesics play important roles. As an application of ideas in the proof of Theorem \ref{Thm:TimeConstStrict}, we can obtain a new result on the expected size of intersection of geodesics for the {\it critical} case. It will be discussed in Section \ref{sec:Intersection}.

\section{Proofs}\label{sec:Proofs}

\subsection{Proof of Theorem \ref{Thm:TimeConstStrict}}

Fix $n \in \mathbb{N}$, $\boldsymbol{x} \in \mathbb{Z}^d$, and $N \in \mathbb{N}$ with $N>\|n\boldsymbol{x}\|_{\infty}$. The first-passage time over paths from $\boldsymbol{O}$ to $n\boldsymbol{x}$ contained in $S(N):=[-N,N]^d$ is denoted by $T^N(\boldsymbol{O},n\boldsymbol{x})$. A path $\gamma$ from $\boldsymbol{O}$ to $n\boldsymbol{x}$ with $t(\gamma)=T^N(\boldsymbol{O},n\boldsymbol{x})$ is called a geodesic for $T^N(\boldsymbol{O},n\boldsymbol{x})$. The edge intersection of geodesics for $T^N(\boldsymbol{O},n\boldsymbol{x})$ is denoted by $\overline{\text{\scshape{Geo}}}^N(\boldsymbol{O},n\boldsymbol{x})$. 

\begin{lemma} \label{lem:PassageTimeRusso} For $p \in (0,1)$,
\begin{align}
-\dfrac{d}{dp} E_p[T^N(\boldsymbol{O},n\boldsymbol{x})] = \dfrac{1}{1-p} E_p[\#\{ e \in \overline{\text{\scshape{Geo}}}^N(\boldsymbol{O},n\boldsymbol{x}): t(e)=1 \} ]. 
\label{eq:PassageTimeRusso}
\end{align}\end{lemma}

\begin{proof} 
Let $A_{n,k} := \{ T^N(\boldsymbol{O},n\boldsymbol{x}) \geq k \}$. We have
\begin{align}
E_p[T^N(\boldsymbol{O},n\boldsymbol{x})] = \sum_{k=1}^{n\|\boldsymbol{x}\|_1} P_p(A_{n,k}).
\label{eq:WuFeng09-1}
\end{align}
The event $A_{n,k}$ depends on the edges in $S(N)$, and decreasing in the configuration of passage times. Define
\begin{align*}
\mathcal{N}(A_{n,k}) := \#\{ e \subset S(N) : \mbox{$e$ is pivotal for $A_{n,k}$} \}.
\end{align*}
By Russo's formula (see e.g. Theorem 2.25 in \cite{Grimmett99Book}), 
\begin{align*}
-\dfrac{d}{dp} P_p(A_{n,k}) &= E_p[ \mathcal{N}(A_{n,k}) ] \\
&= \sum_{e \subset S(N)} P_p(\mbox{$e$ is pivotal for $A_{n,k}$}) \\
&= \dfrac{1}{1-p} \sum_{e \subset S(N)} P_p(\{\mbox{$e$ is pivotal for $A_{n,k}$}\} \cap A_{n,k}) \\
&= \dfrac{1}{1-p} E_p[ \mathcal{N}(A_{n,k}) : A_{n,k}].
\end{align*}
Noting that
\begin{align*}
 \mathcal{N}(A_{n,k}) = \begin{cases}
  \#\{ e \in \overline{\text{\scshape{Geo}}}^N(\boldsymbol{O},n\boldsymbol{x}): t(e)=1 \} &\mbox{if $T^N(\boldsymbol{O},n\boldsymbol{x})=k$, and} \\
  0 & \mbox{if $T^N(\boldsymbol{O},n\boldsymbol{x})>k$},
  \end{cases}
\end{align*}
we have
\begin{align}
&-\dfrac{d}{dp} P_p(A_{n,k}) \notag \\
&= \dfrac{1}{1-p} E_p[ \#\{ e \in \overline{\text{\scshape{Geo}}}^N(\boldsymbol{O},n\boldsymbol{x}): t(e)=1 \} : T^N(\boldsymbol{O},n\boldsymbol{x})=k].
\label{eq:WuFeng09-2}
\end{align}
Now \eqref{eq:PassageTimeRusso} follows from \eqref{eq:WuFeng09-1} and \eqref{eq:WuFeng09-2}.
\end{proof}

Since 
\begin{align}
\#\{ e \in \overline{\text{\scshape{Geo}}}^N(\boldsymbol{O},n\boldsymbol{x}): t(e)=1 \} \geqq t(\gamma)=T^N(\boldsymbol{O},n\boldsymbol{x})
\label{ineq:FengWu09Key}
\end{align}
for any geodesic $\gamma$ for $T^N(\boldsymbol{O},n\boldsymbol{x})$, we have
\[  -\dfrac{d}{dp} E_p[T^N(\boldsymbol{O},n\boldsymbol{x})] \geq \dfrac{1}{1-p} \cdot E_p[T^N(\boldsymbol{O},n\boldsymbol{x})],  \]
which implies that
\[ \dfrac{d}{dp} \left( \dfrac{E_p[T^N(\boldsymbol{O},n\boldsymbol{x})]}{1-p} \right) \leq 0\quad \mbox{for $p \in (0,1)$.} \]
Thus we obtain
\[ \dfrac{E_p[T^N(\boldsymbol{O},n\boldsymbol{x})]}{1-p} \geq \dfrac{E_q[T^N(\boldsymbol{O},n\boldsymbol{x})]}{1-q} \quad \mbox{for $0\leq p<q<1$.} \]
Letting $N \to \infty$, dividing both sides by $n$, and then letting $n \to \infty$, we have
\[ \dfrac{\mu_p(\boldsymbol{x})}{1-p} \geq \dfrac{\mu_q(\boldsymbol{x})}{1-q}, \]
which is equivalent to \eqref{Ineq:TimeConstStrict}.
\qed

\subsection{Proof of Theorem \ref{Thm:TimeConstLipschitz}}

Our proof is inspired by that of Lemma 3 in Zhang \cite{Zhang06PTRF}. We begin with a lemma. 
The geodesic for $T(\boldsymbol{O},\boldsymbol{x})$ with maximal number of edges is denoted by $\text{\scshape{Geo}}(\boldsymbol{O},\boldsymbol{x})$.

\begin{lemma} \label{lem:ShotestGeoUpperBound} Let $q_0 \in (0,p_c)$. There exists a constant $K=K(q_0)>0$ such that for any $q \in (0, q_0]$ and any $\boldsymbol{x} \in \mathbb{Z}^d$,
\begin{align}
E_q[\# \text{\scshape{Geo}}(\boldsymbol{O},\boldsymbol{x})] \leq K \cdot E_q[T(\boldsymbol{O},\boldsymbol{x})].
\label{ineq:ShotestGeoUpperBound}
\end{align}
\end{lemma} 

\begin{proof} We closely follows that of Proposition 4.7 (1) in Auffinger, Damron and Hanson \cite{AuffingerDamronHanson50years}. There exist constants $a,C_1>0$ depending only on $q_0$ such that for any $q \in (0, q_0]$ and any $m \in \mathbb{N}$,
\begin{align}
 P_q\left(
\begin{array}{@{\,}c@{\,}}
\text{there is a self-avoiding path $\gamma$ stating from $\boldsymbol{O}$}\\
\text{with $\# \gamma \geq m$ but $t(\gamma) \leq a m$} \\
\end{array}
\right) \leqq e^{-C_1 m}.
\label{ineq:Kesten86Prop5.8Bernoulli}
\end{align}
To prove this, we have only to choose the constant $M$ in the proof of Lemma 4.5 in \cite{AuffingerDamronHanson50years} as 
\[
\sum_{\boldsymbol{y} \in \mathbb{Z}^d : \| \boldsymbol{y} \|_{\infty}=M} P_{q_0} (T(\boldsymbol{O},\boldsymbol{y})=0) \leq \dfrac{1}{2}.
\]
We estimate
\begin{align*}
E_q[\# \text{\scshape{Geo}}(\boldsymbol{O},\boldsymbol{x})] \leq \dfrac{1}{a} E_q[T(\boldsymbol{O},\boldsymbol{x})] + E_q[Y_{\boldsymbol{x}}], 
\end{align*}
where 
\[ Y_{\boldsymbol{x}} := \# \text{\scshape{Geo}}(\boldsymbol{O},\boldsymbol{x}) \cdot 1_{\{ T(\boldsymbol{O},\boldsymbol{x}) < a \# \text{\scshape{Geo}}(\boldsymbol{O},\boldsymbol{x})\}}. \]
Following the proof of Lemma 4.6 in \cite{AuffingerDamronHanson50years} together with \eqref{ineq:Kesten86Prop5.8Bernoulli}, we can find a constant $C_2=C_2(q_0)>0$ such that
\[ P_q(Y_{\boldsymbol{x}} \geq n) \leq C_2 e^{-C_1n} \quad \mbox{for $n \in \mathbb{N}$}, \]
which implies that
\[ E_q[Y_{\boldsymbol{x}}] \leq C_3 \]
for some $C_3=C_3(q_0)>0$.
Since $E_{q_0}[T(\boldsymbol{O},\boldsymbol{x})] \to \infty$ as $\|\boldsymbol{x}\|_1 \to \infty$,
we can take a constant $C_4=C_4(q_0)$ such that
\[
C_3 \leq C_4 \cdot E_{q_0}[T(\boldsymbol{O},\boldsymbol{x})] \quad \mbox{for all $\boldsymbol{x} \in \mathbb{Z}^d$.}
\]
Noting that $E_q[T(\boldsymbol{O},\boldsymbol{x})] \geq E_{q_0}[T(\boldsymbol{O},\boldsymbol{x})]$, \eqref{ineq:ShotestGeoUpperBound} holds true with $K:=\dfrac{1}{a} +C_4$.
\end{proof}

Let $\{U(e)\}_{e \in \mathbb{E}^d}$ be independent uniform random variables on $[0,1]$, and define
\[ t_p(e) := \begin{cases}
0 &(U(e) \leq p), \\
1 &(U(e)>p) \\
\end{cases} \quad \mbox{for each $e \in \mathbb{E}^d$.} \]
Suppose that $0<p<q<p_c$ and $\boldsymbol{x} \in \mathbb{Z}^d$. The first-passage time from $\boldsymbol{O}$ to $n\boldsymbol{x}$ in the configuration $\{t_p(e)\}_{e \in \mathbb{E}^d}$ is denoted by $T_p(\boldsymbol{O},n\boldsymbol{x})$. A geodesic from $\boldsymbol{O}$ to $n\boldsymbol{x}$ in the configuration $\{t_q(e)\}_{e \in \mathbb{E}^d}$ with maximal number is denoted by $\text{\scshape{Geo}}_q (\boldsymbol{O},n\boldsymbol{x})$. Then we have
\begin{align}
T_p(\boldsymbol{O},n\boldsymbol{x}) &\leq \sum_{b \in \text{\scshape{Geo}}_q (\boldsymbol{O},n\boldsymbol{x})} t_p(b) \notag \\
&=  \sum_{b \in \text{\scshape{Geo}}_q (\boldsymbol{O},n\boldsymbol{x})} t_q(b) +  \sum_{b \in \text{\scshape{Geo}}_q (\boldsymbol{O},n\boldsymbol{x})} \{ t_p(b)-t_q(b) \} \notag \\
&= T_q(\boldsymbol{O},n\boldsymbol{x}) + \sum_{b \in \text{\scshape{Geo}}_q (\boldsymbol{O},n\boldsymbol{x})} \{ t_p(b)-t_q(b) \}. \label{ineq:Zhang06Lem3var1}
\end{align}
We prove
\begin{align}
E\left[ \sum_{b \in \text{\scshape{Geo}}_q (\boldsymbol{O},n\boldsymbol{x})} \{ t_p(b)-t_q(b) \}  \right] \leq \dfrac{q-p}{q} \cdot E_q[\# \text{\scshape{Geo}} (\boldsymbol{O},n\boldsymbol{x})].
\label{ineq:Zhang06Lem3var2}
\end{align}
The left hand side of \eqref{ineq:Zhang06Lem3var2} is
\begin{align}
 \sum_{\gamma: \boldsymbol{O} \to n\boldsymbol{x}} \sum_{b \in \gamma} P(t_p(b)=1,\,t_q(b)=0,\text{\scshape{Geo}}_q (\boldsymbol{O},n\boldsymbol{x})=\gamma).
 \label{ineq:Zhang06Lem3var3}
\end{align}
We define a new configuration $\{t_q^*(e)\}_{e \in \mathbb{E}^d}$ by
\[
t_q^*(e):=\begin{cases}
0 &(e=b), \\
t_q(e) &(e \neq b) \\
\end{cases}  \quad \mbox{for each $e \in \mathbb{E}^d$.}
\]
A geodesic from $\boldsymbol{O}$ to $n\boldsymbol{x}$ in the configuration $\{t_q^*(e)\}_{e \in \mathbb{E}^d}$ with maximal number is denoted by $\text{\scshape{Geo}}_q^* (\boldsymbol{O},n\boldsymbol{x})$.
Using independence, 
\begin{align*}
&P(t_p(b)=1 \mid t_q(b)=0,\text{\scshape{Geo}}_q (\boldsymbol{O},n\boldsymbol{x})=\gamma) \\
&=P(t_p(b)=1 \mid t_q(b)=0,\text{\scshape{Geo}}_q^* (\boldsymbol{O},n\boldsymbol{x})=\gamma) \\
&=P(t_p(b)=1 \mid  t_q(b)=0) = \dfrac{q-p}{q}.
\end{align*}
As
\[
P(t_q(b)=0,\text{\scshape{Geo}}_q (\boldsymbol{O},n\boldsymbol{x})=\gamma)\leq P_q (\text{\scshape{Geo}} (\boldsymbol{O},n\boldsymbol{x})=\gamma),
\]
we can see that \eqref{ineq:Zhang06Lem3var3} is not more than
\begin{align*}
&\sum_{\gamma: \boldsymbol{O} \to n\boldsymbol{x}} \sum_{b \in \gamma} \dfrac{q-p}{q} \cdot P_q (\text{\scshape{Geo}} (\boldsymbol{O},n\boldsymbol{x})=\gamma) \\
&= \dfrac{q-p}{q} \cdot \sum_{\gamma: \boldsymbol{O} \to n\boldsymbol{x}} (\# \gamma)\cdot P_q (\text{\scshape{Geo}} (\boldsymbol{O},n\boldsymbol{x})=\gamma) \\
&= \dfrac{q-p}{q} \cdot E_q[\# \text{\scshape{Geo}} (\boldsymbol{O},n\boldsymbol{x})].
\end{align*}
This completes the proof of \eqref{ineq:Zhang06Lem3var2}. By \eqref{ineq:Zhang06Lem3var1} and \eqref{ineq:Zhang06Lem3var2}, we obtain
\begin{align*}
E_p[T(\boldsymbol{O},n\boldsymbol{x}) ] &\leq E_q[T(\boldsymbol{O},n\boldsymbol{x})]+ \dfrac{q-p}{q} \cdot E_q[ \#\text{\scshape{Geo}} (\boldsymbol{O},n\boldsymbol{x}) ].
\end{align*}
\eqref{Ineq:TimeConstLipschitz} follows from Lemma \ref{lem:ShotestGeoUpperBound} and \eqref{def:TimeConstGen}.
\qed

\begin{proof}[Proof of Corollary \ref{Cor:AsymptoticShapeHaudorff}] Assume that $0 \leq p<q <p_c$. Since $\mathcal{B}_p \subset \mathcal{B}_q$, 
\[
d_{\mathcal{H}}(\mathcal{B}_p, \mathcal{B}_q) = \inf\{ \delta>0 : \text{$\mathcal{B}_q \subset  N_{\delta} (\mathcal{B}_p)$} \}.
\]
Define
\[ \mathbb{S}^{d-1}:=\{ \boldsymbol{u} \in \mathbb{R}^d : \|  \boldsymbol{u} \|_2 = 1\}. \]
We can see that
\[ \inf_{\boldsymbol{u} \in\mathbb{S}^{d-1}} \left\| \dfrac{\boldsymbol{u}}{\mu_q(\boldsymbol{u})} - \dfrac{\boldsymbol{u}}{\mu_p(\boldsymbol{u})}\right\|_2 \leq d_{\mathcal{H}}(\mathcal{B}_p,\mathcal{B}_q) \leq \sup_{\boldsymbol{u} \in\mathbb{S}^{d-1}} \left\| \dfrac{\boldsymbol{u}}{\mu_q(\boldsymbol{u})} - \dfrac{\boldsymbol{u}}{\mu_p(\boldsymbol{u})}\right\|_2, \]
and
\begin{align*}
\left\| \dfrac{\boldsymbol{u}}{\mu_q(\boldsymbol{u})} - \dfrac{\boldsymbol{u}}{\mu_p(\boldsymbol{u})}\right\|_2 &= \left|  \dfrac{1}{\mu_q(\boldsymbol{u})} -  \dfrac{1}{\mu_p(\boldsymbol{u})} \right| \cdot \| \boldsymbol{u}\|_2 
= \dfrac{\mu_p(\boldsymbol{u})-\mu_q(\boldsymbol{u})}{\mu_p(\boldsymbol{u})\mu_q(\boldsymbol{u})}.
\end{align*}
By Theorem \ref{Thm:TimeConstStrict},
\begin{align*}
d_{\mathcal{H}}(\mathcal{B}_p, \mathcal{B}_q) &\geq \dfrac{q-p}{1-q} \cdot \inf_{\boldsymbol{u} \in\mathbb{S}^{d-1}} \dfrac{1}{\mu_p(\boldsymbol{u})} \\
 &\geq (q-p) \cdot \inf_{\boldsymbol{u} \in\mathbb{S}^{d-1}} \dfrac{1}{\mu_0(\boldsymbol{u})} = q-p,
\end{align*}
where we used $\mu_0(\boldsymbol{u}) = \|\boldsymbol{u}\|_1$. On the other hand, by Theorem \ref{Thm:TimeConstLipschitz},
\begin{align*}
d_{\mathcal{H}}(\mathcal{B}_p, \mathcal{B}_q) &\leq \dfrac{K(q-p)}{q} \cdot \sup_{\boldsymbol{u} \in\mathbb{S}^{d-1}} \dfrac{1}{\mu_p(\boldsymbol{u})} 
 \leq \dfrac{K(q-p)}{p_0} \cdot \sup_{\boldsymbol{u} \in\mathbb{S}^{d-1}} \dfrac{1}{\mu_{q_0}(\boldsymbol{u})}.
\end{align*}
\end{proof}

\section{A remark on the intersection of geodesics} \label{sec:Intersection}

Let $\boldsymbol{x} \in \mathbb{Z}^d$. 
The proof of (4.10) in Auffinger, Damron and Hanson \cite{AuffingerDamronHanson50years} contains the following inequality:
\begin{align}
  E_p\left[ \# \underline{\text{\scshape{Geo}}} (\boldsymbol{O},\boldsymbol{x})\right] \geq \frac{p}{1-p} \cdot E_p[\#\{ e \in \overline{\text{\scshape{Geo}}}(\boldsymbol{O},\boldsymbol{x}): t(e)=1 \} ]
  \label{ineq:AuffingerDamronHanson17(4.10)Bernoulli}
\end{align}
holds for {\it any} $p \in (0,1)$ (set $I=0$ and $\delta=1/4$ in that proof). Similarly to \eqref{ineq:FengWu09Key}, we can deduce 
that
\begin{align}
E_p\left[ \# \underline{\text{\scshape{Geo}}} (\boldsymbol{O},\boldsymbol{x})\right] \geq \frac{p}{1-p} \cdot E_p[T(\boldsymbol{O},\boldsymbol{x})].
 \label{ineq:AuffingerDamronHanson17(4.10)Bernoulli_2}
\end{align}

For $p \in [0,p_c)$, it follows from \eqref{ineq:AuffingerDamronHanson17(4.10)Bernoulli_2}, \eqref{ineq:ShotestGeoUpperBound} and \eqref{def:TimeConstGen} that there exists a constant $C=C(p)>0$ such that
\begin{align}
\dfrac{p}{1-p} \cdot \mu_p( \boldsymbol{x})n \leq E_p\left[ \# \underline{\text{\scshape{Geo}}} (\boldsymbol{O},n\boldsymbol{x})\right] \leq  C n\|  \boldsymbol{x} \|_1
\label{ineq:AuffingerDamronHanson17(4.10)BernoulliTwoSided}
\end{align}
for any $\boldsymbol{x} \in \mathbb{Z}^d$ and $n \in \mathbb{N}$. Nakajima \cite{Nakajima19JSP} proves that $E_p\left[ \# \underline{\text{\scshape{Geo}}} (\boldsymbol{O},n\boldsymbol{e}_1)\right]$ is of order $n$ if $t(e)$ is square-integrable and satisfies an additional condition called ``useful". Our proof of \eqref{ineq:AuffingerDamronHanson17(4.10)BernoulliTwoSided} based on \eqref{ineq:AuffingerDamronHanson17(4.10)Bernoulli_2} gives a simple alternative proof for the Bernoulli case. In view of \eqref{ineq:AuffingerDamronHanson17(4.10)Bernoulli}, we can see that $E_p[\#\{ e \in \overline{\text{\scshape{Geo}}}(\boldsymbol{O},n\boldsymbol{x}): t(e)=1 \} ]$ is also of order $n$, which is a key for proving that $E_p[\overline{\text{\scshape{Geo}}}(\boldsymbol{O},n\boldsymbol{x})]$ is of order $n$ (See Theorem 2 in Zhang \cite{Zhang06PTRF}, Theorem 4.12 in Auffinger, Damron and Hanson \cite{AuffingerDamronHanson50years}, and Theorem 3 in Nakajima \cite{Nakajima19JSP}).

We emphasize that \eqref{ineq:AuffingerDamronHanson17(4.10)Bernoulli_2} gives an interesting information about the critical case. Let $\partial S(n):=\{ \boldsymbol{x} \in \mathbb{Z}^d : \|\boldsymbol{x}\|_{\infty}= n \}$, and the first-passage time from $\boldsymbol{O}$ to $\partial S(n)$ is denoted by $T(\boldsymbol{O},\partial S(n))$. Under the hypothesis
\begin{align}
P_{p_c} 
\left(
\begin{array}{@{\,}c@{\,}}
\text{there exists an infinite self-avoiding path}\\
\text{of edges $e$ with $t(e)=0$} \\
\end{array}
\right)=0,
\label{hypothesis:NoPercolationAtCriticality}
\end{align}
we have
\begin{align*}
P_{p_c} \left(\lim_{n \to \infty} T(\boldsymbol{O},\partial S(n)) = + \infty\right)=1,
\end{align*}
and \eqref{ineq:AuffingerDamronHanson17(4.10)Bernoulli_2} implies that
\begin{align*}
\lim_{n \to \infty} E_{p_c}\left[ \# \underline{\text{\scshape{Geo}}} (\boldsymbol{O},n\boldsymbol{x})\right] = +\infty.
\end{align*}
This suggests that the size of the intersection of geodesics can be large even if $p=p_c$. The verification of the hypothesis \eqref{hypothesis:NoPercolationAtCriticality} for all $d \geq 2$ is a long-standing problem in the percolation theory. At the moment \eqref{hypothesis:NoPercolationAtCriticality} is shown to be true for $d=2$ and for sufficiently large $d$ (see \cite{HeydenreichRemco17Book} and \cite{Grimmett99Book}). 
Chayes, Chayes, and Durrett \cite{ChayesChayesDurrett86JSP} prove that there exist constants $C_1,C_2>0$ such that
\[ C_1 \log n \leq E_{p_c} [ T(\boldsymbol{O},\partial S(n))] \leq C_2 \log n \quad \mbox{for $n \geq 3$}. \]
On the other hand, Zhang \cite{Zhang99Double} shows that for sufficiently large $d$, there exists a constant  $C>0$ such that
\[E_{p_c}[T(\boldsymbol{O},\partial S(n))]  \geq C \log \log n \quad \mbox{for $n \geq 3$}. \]
Those results give lower bounds of $E_{p_c}\left[ \# \underline{\text{\scshape{Geo}}} (\boldsymbol{O},n\boldsymbol{x})\right]$. One of important future problems is to give the precise order for 
$E_{p_c}\left[ \# \underline{\text{\scshape{Geo}}} (\boldsymbol{O},n\boldsymbol{x})\right]$ and $E_{p_c}\left[ \# \overline{\text{\scshape{Geo}}} (\boldsymbol{O},n\boldsymbol{x})\right]$. It is proved in Damron and Tang \cite{DamronTang19NewmanVolume} that when $d=2$, there exist constants $c>0$ and $s>1$ such that 
for any $\boldsymbol{x} \in \mathbb{Z}^2 \setminus \{ \boldsymbol{O}\} $, 
\[
P_{p_c}(N_{\boldsymbol{O},\boldsymbol{x}} \leq \| \boldsymbol{x}\|_1^s ) \leq \dfrac{1}{c}e^{-\| \boldsymbol{x}\|_1^c},
\]
where $N_{\boldsymbol{O},\boldsymbol{x}}$ is the minimal number of edges in any geodesic from $\boldsymbol{O}$ to $\boldsymbol{x}$.



\begin{thebibliography}{99}
%
\bibitem{AuffingerDamronHanson50years}
Auffinger, A., Damron, M., and Hanson, J. (2017). 50 years of first-passage percolation, {\it University Lecture Series}, {\bf 68}, American Mathematical Society.
\bibitem{vandenBergKesten93AOAP}
van den Berg, J. and Kesten, H. (1993). Inequalities for the time constant in first-passage percolation, {\it Ann. Appl. Probab.}, {\bf 3}, 56--80. 
\bibitem{CerfDembin21}
Cerf, R. and Dembin, B. (2021). The time constant for Bernoulli percolation is Lipschitz continuous strictly above $p_c$, arXiv:2101.11858, to appear in {\it Ann. Probab.}
\bibitem{ChayesChayesDurrett86JSP}
Chayes, J.T., Chayes, L., and Durrett, R. (1986). Critical behavior of two-dimensional first-passage times, {\it J. Stat. Phys.}, {\bf 45}, 933--951.
\bibitem{Cox80}
Cox, J.T. (1980). The time constant of first-passage percolation on the square lattice, {\it Adv. in Appl. Probab.}, {\bf 12}, 864--879.
\bibitem{CoxKesten81}
Cox, J.T. and Kesten, H. (1981). On the continuity of the time constant of first-passage percolation, {\it J. Appl. Probab.}, {\bf 18}, 809--819.  
\bibitem{DamronTang19NewmanVolume} Damron, M. and Tang, P. (2019). Superlinearity of geodesic length in 2D critical first-passage percolation, {\it Sojourns in probability theory and statistical physics. II. Brownian web and percolation, a festschrift for Charles M. Newman}, {\it Proc. Math. Stat.}, {\bf 299}, 101--122, Springer.  
\bibitem{Dembin21ESAIM}
Dembin, B. (2021). Regularity of the time constant for a supercritical Bernoulli percolation, {\it ESAIM, Probab. Stat.}, {\bf 25}, 109--132.
\bibitem{HammersleyWelsh65}
Hammersley, J.M. and Welsh, D.J.A. (1965). First-passage percolation, subadditive processes, stochastic networks, and generalized renewal theory, {\it Bernoulli, Bayes, Laplace Anniversary volume}, 61--110, Springer.
\bibitem{HeydenreichRemco17Book}
Heydenreich, M. and van der Hofstad, R. (2017). Progress in high-dimensional percolation and random graphs, {\it CRM Short Courses}, Springer.
\bibitem{GaretMarchandProcacciaTheret17EJP}
Garet, O., Marchand, R., Procaccia, E.B., and Th\'{e}ret, M. (2017). Continuity of the time and isoperimetric constants in supercritical percolation, {\it Electron. J. Probab.}, {\bf 22}, Paper No. 78, 35 p. 
\bibitem{Grimmett99Book}
Grimmett, G. (1999). Percolation, Second edition, {\it Grundlehren math. Wissenschaften}, {\bf 321}, Springer.
\bibitem{Kesten86LNM}
Kesten, H. (1986). Aspects of first passage percolation, {\it \'{E}cole d'\'{e}t\'{e} de probabilit\'{e}s de Saint-Flour, XIV -- 1984}, {\it Lecture Notes in Math.}, {\bf 1180}, 125--264, Springer.
\bibitem{Nakajima19JSP}
Nakajima, S. (2019). On properties of optimal paths in first-passage percolation, {\it J. Stat. Phys.}, {\bf 174}, 259--275. 
\bibitem{WuFeng09}
Wu, X.-Y. and Feng, P. (2009). A lower bound for the time constant of first-passage percolation, {\it Acta Math. Sinica (Chin. Ser.)}, {\bf 52}, 495--500; (2008). English version, arXiv:0807.0839
\bibitem{Zhang99Double}
Zhang, Y. (1999). Double behavior of critical first-passage percolation, {\it Perplexing problems in probability}, {\it Progr. Probab.}, {\bf 44}, 143--158, Birkh\"{a}user. 
\bibitem{Zhang06PTRF}
Zhang, Y. (2006). The divergence of fluctuations for the shape of first passage percolation, {\it Probab. Theory Relat. Fields}, {\bf 136}, 298--320. 
\end{thebibliography}


\end{document}